\pdfoutput=1  

\documentclass{amsart}
\usepackage{amssymb}
\usepackage{amsmath}

\usepackage{graphicx}

\usepackage{url}

\newtheorem{theorem}{Theorem}[section]
\newtheorem{lemma}[theorem]{Lemma}

\newtheorem{question}[theorem]{Question}

\theoremstyle{definition}

\theoremstyle{remark}
\newtheorem{remark}[theorem]{Remark}

\numberwithin{equation}{section}

\begin{document}

\title[L-space knots with braid index four and tunnel number two]{On hyperbolic L--space knots with braid index four and tunnel number two}


\author[M. Teragaito]{Masakazu Teragaito}
\address{Department of Mathematics Education, Hiroshima University,
1-1-1 Kagamiyama, Higashi-hiroshima 7398524, Japan.}
\email{teragai@hiroshima-u.ac.jp}
\thanks{
The author has been partially supported by JSPS KAKENHI Grant Number JP25K07004.}

\subjclass[2020]{Primary 57K10}

\date{\today}



\begin{abstract}
There are only three known strongly invertible hyperbolic L--space knots with
braid index four and tunnel number two.
They are \texttt{t09284}, \texttt{t10496} and \texttt{o9\_34409} in the SnapPy census.
In this paper, we give the first infinite family of strongly invertible hyperbolic
L--space knots with braid index four and tunnel number two that
includes \texttt{t10496} and \texttt{o9\_34409}.
\end{abstract}

\keywords{L--space knot, braid index, tunnel number}
\maketitle


\section{Introduction}\label{sec:intro}

A knot  in the $3$-sphere $S^3$ is called an \textit{L--space knot\/} if it admits a positive Dehn surgery yielding an L--space.
A typical example is a (positive) torus knot.
In this paper, we focus on the braid index of L--space knots.
Knots with braid index two are torus knots.
It is known that L--space knots  with braid index three
are either torus knots or braid positive twisted torus knots \cite{LV}.
(Here, a knot that is the closure of a positive braid is said to be braid positive.)
Also,  L--space knots with braid index less than four have tunnel number one.

In the SnapPy census (\cite{ABG, BK, D1,D2}), there are 632 L--space knots, and there are only three known strongly invertible
hyperbolic L--space knots with braid index four and tunnel number two; 
\texttt{t09284}, \texttt{t10496} and \texttt{o9\_34409} (\cite{ABG}).
This observation leads to the following question.

\begin{question}[{\cite[Question 10]{ABG}}]
What are the L--space knots of braid index four with tunnel number greater than one?
\end{question}

In a previous paper \cite{BT},  we gave an infinite family of asymmetric hyperbolic L--space knots with braid index four.
We can verify that these knots have tunnel number two.
(It is not hard to give an unknotting tunnel system consisting of  two arcs.
Then the asymmetry excludes the possibility of tunnel number one, because a tunnel number one knot is strongly invertible \cite[Lemma 5]{Mor}.)

The purpose of this paper is to give the first infinite family of strongly invertible hyperbolic L--space knots with
braid index four and tunnel number two.

\begin{theorem}\label{thm:main}
There are infinitely many strongly invertible, hyperbolic L--space knots
with braid index four and tunnel number two.
\end{theorem}

For an integer $n\ge 0$, let $K_n$ be the closure of the positive $4$-braid
\begin{equation*}
[3,2,2,1,3,2,2,3,2, (1,2)^{3n}],
\end{equation*}
where an integer $i$ denotes the standard braid generator $\sigma_i$ in the braid group $B_4$ of four strands.
Since it is the closure of a positive braid, it is fibered (\cite[Theorem 2]{S}), and it has genus $3n+3$.
Hence $K_m\ne K_n$ if $m\ne n$.
Note that $K_0$ is  the torus knot $T(3,4)$, and  $K_1$ is the $(2,9)$-cable of $T(2,3)$.
Also, $K_2$, $K_3$ and $K_4$ are \texttt{s682}, \texttt{t10496} and \texttt{o9\_34409}, respectively, in the census.

Theorem \ref{thm:main} immediately follows from the following.

\begin{theorem}\label{thm:main2}
Let $K_n$ be the knots defined as above.
Then $K_n$ satisfies the following.
\begin{itemize}
\item[(1)]
$K_n$ is a strongly invertible L--space knot.
\item[(2)]
If $n\ge 1$, then $K_n$ has braid index four.
\item[(3)]
If $n\ge 2$, then $K_n$ is hyperbolic.
\item[(4)]
If $n\ge 3$, then $K_n$ has tunnel number two.
\end{itemize}
\end{theorem}

We have two other candidates of infinite families which generalize \texttt{t09284}.
They are the closures of $4$-braids
\[
[(1,2,3)^4,2,1,3,2,2,3^{2m}], \quad 
[(1,2,3)^{4n}, 2,1,3,2,2,3^6].
\]
When $m=3$ or $n=1$, each gives \texttt{t09284}.
However, we could not prove that these knots have tunnel number two.

\section{L--space knots}\label{sec:Lspace}

In this section, we prove that our knot $K_n$ is a strongly invertible L--space knot.
Let $K_0\cup c$ be the link as shown in Figure \ref{fig:str}.
By performing $(-1/n)$-surgery along $c$, $K_0$ will be changed into $K_n$.

\begin{figure}[ht]
\includegraphics*[width=0.5 \textwidth]{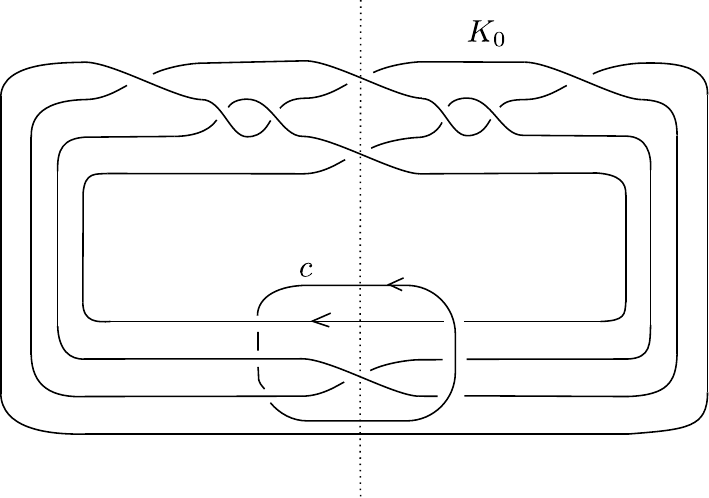}
\caption{The link $K_0\cup c$.  By performing $(-1/n)$-surgery along $c$, 
$K_0$ will be changed into $K_n$.  This diagram  also shows a strongly invertible position of the link.}
\label{fig:str}
\end{figure}
 
\begin{lemma}\label{lem:strong}
For any $n\ge 0$, $K_n$ is strongly invertible.
\end{lemma}

\begin{proof}
Figure \ref{fig:str} shows a strongly invertible position of the link $K_0\cup c$.
The dotted line is the axis of the involution.
Performing $(-1/n)$-surgery on $c$ changes $K_0$ into $K_n$.
This observation confirms that $K_n$ is strongly invertible for any $n\ge 0$.
\end{proof}

\begin{lemma}\label{lem:kirby}
For the link $K_0\cup c$, 
$(8,0)$-surgery yields a lens space $L(9,7)$.
\end{lemma}

\begin{proof}
This is confirmed by Kirby--Rolfsen calculus.
See Figure \ref{fig:kirby}.
\end{proof}

\begin{figure}[ht]
\includegraphics*[width=\textwidth]{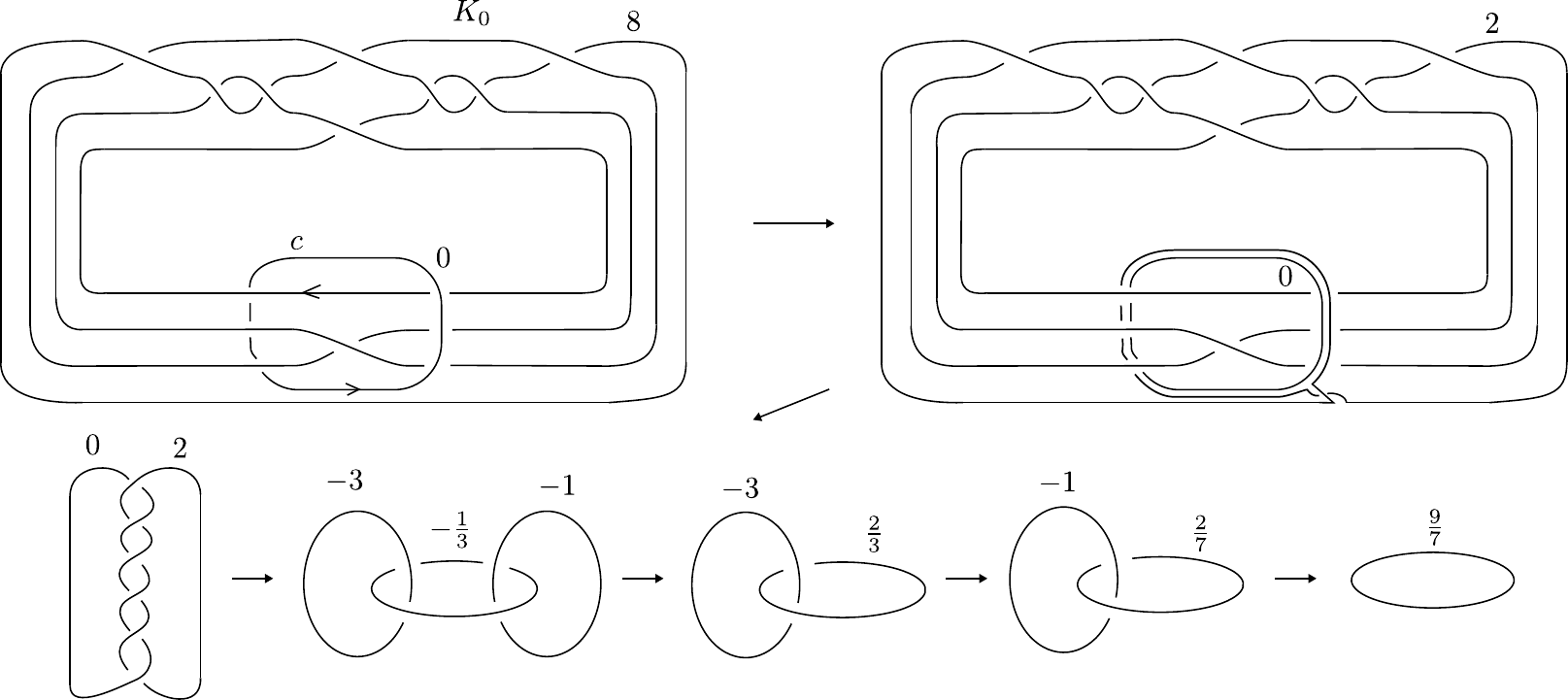}
\caption{Kirby--Rolfsen calculus for $(8,0)$-surgery on $K_0\cup c$.
The result is a lens space $L(9,7)$.}
\label{fig:kirby}
\end{figure}

\begin{lemma}[{\cite[Theorem 1.13 and Lemma 6.1]{BM}, \cite[Lemma 2.1]{BT}}]\label{lem:Lspace-interval}
Let $K\cup c$ be a link of an L--space knot $K$ and an unknot $c$ with linking number $w=\mathrm{lk}(K,c)>1$.
Let $K_n$ be the image of $K$ after $(-1/n)$-surgery on $c$ for integers $n$.
If there exists a slope $r\ge 2g(K)-1$ such that $(r,0)$-surgery on $K\cup c$ is an L--space, then $K_n$ is an L--space knot
for all integers $n\ge 0$.
\end{lemma}

\begin{lemma}\label{lem:Lspace}
For each $n\ge 0$, $K_n$ is an L--space knot.
\end{lemma}

\begin{proof}
First, $K_0$ is a torus knot $T(3,4)$ of genus three.
By Lemma \ref{lem:kirby}, $(8,0)$-surgery on $K_0\cup c$ 
yields an L--space, because a lens space is an L--space.
Thus our link $K_0\cup c$ satisfies the assumption of Lemma \ref{lem:Lspace-interval}.
\end{proof}

\section{Tunnel number}\label{sec:t2}

In this section, we prove that $K_n$ has tunnel number two if $n\ge 3$.
We need a lemma using the Montesinos trick \cite{M}.

\begin{lemma}\label{lem:graph}
If $n\ge 3$, then 
$(9n+9)$-surgery on $K_n$ yields a graph manifold $M$ which
is the union of two Seifert fibered spaces over the disk with two exceptional fibers
of indices $(3,3)$ and $(2,n-1)$, respectively.
In addition, the graph manifold $M$  has Heegaard genus at least three.
\end{lemma}

\begin{proof}
For the link $K_0\cup c$ given in Figure \ref{fig:str}, consider $(9,-1/n)$-surgery.
Performing $n$-twists on $c$ shows that it represents $(9n+9)$-surgery on $K_n$.

Take the quotient under the involution along the axis depicted there.
Then we have a diagram as shown in Figure \ref{fig:mont} (top left)  after the tangle replacements.
(Since $K_0$ has writhe $9$ in the diagram of Figure \ref{fig:str},
the tangle surgery for the quotient of $K_0$ is done by the $0$-tangle.)

After a series of deformation in Figures \ref{fig:mont} and \ref{fig:mont1},
the resulting link (right in Figure \ref{fig:mont1}) consists of
two Montesinos tangles $[-1/3,1/3]$ and $[-1/2,-1/(n-1)]$.
By taking the double branched cover, which gives $M$, 
we have the first conclusion.

To confirm that this graph manifold $M$ has Heegaard genus at least three, 
we follow the argument in \cite[Subsection 2.3]{ABG},
which is based on the classification of closed orientable $3$-manifolds with Heegaard genus two that
contain an essential torus by Kobayashi \cite{Ko}.
It claims that if a graph manifold obtained from gluing two Seifert fibered manifolds along an incompressible boundary torus
has Heegaard genus two and contains a unique essential torus, then the regular fibers of the two pieces on the common torus intersect just once.

First, the cases where $n=3$ and $4$ are treated in \cite[2.3.2]{ABG}.
Thus we assume $n\ge 5$.
We remark that  each piece of $M$  admits the unique Seifert fibration.
(In fact, this holds when $n\ge 4$.)
From Figure \ref{fig:mont1}, we see that the regular fibers of the two pieces on the common
torus intersect twice (see the proof of \cite[Proposition 12.28]{BZ}).
This implies that $M$ has Heegaard genus at least three. \end{proof}

\begin{figure}[ht]
\includegraphics*[width=\textwidth]{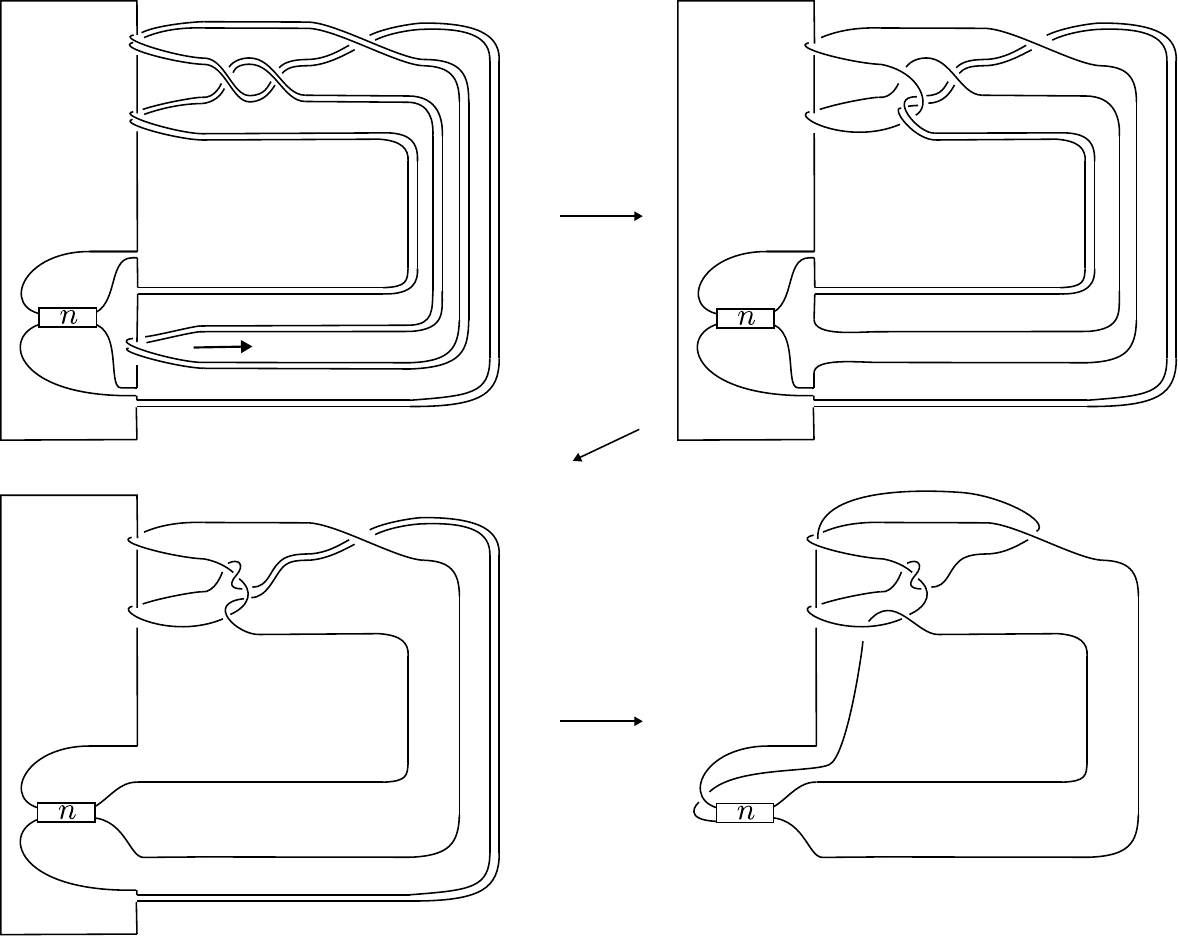}
\caption{ The diagram after the tangle replacements (top left) and its deformation.
The box with integer $n$ contains right handed horizontal $n$ half twists.
}
\label{fig:mont}
\end{figure}

\begin{figure}[ht]
\includegraphics*[width=\textwidth]{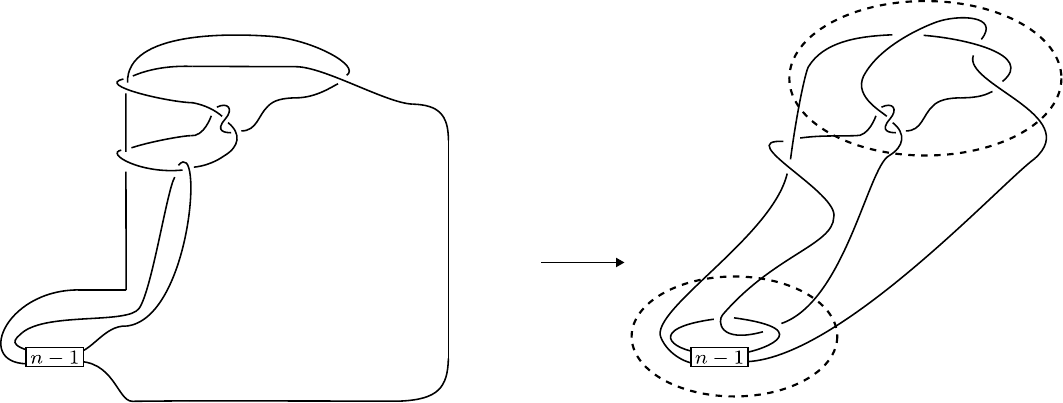}
\caption{Further deformation.
The resulting link (right) consists of two Montesinos tangles $[-1/3,1/3]$ and $[-1/(n-1), -1/2]$.}
\label{fig:mont1}
\end{figure}

\begin{lemma}\label{lem:t2}
If $n\ge 3$, then $K_n$ has tunnel number two.
\end{lemma}

\begin{proof}
First, the two arcs $t_1$ and $t_2$, as shown in Figure \ref{fig:tunnel},  give an unknotting tunnel system for $K_n$.
It is straightforward to verify that the exterior of the neighborhood of $K_n\cup t_1\cup t_2$ is a handlebody of genus three.
Hence $K_n$ has tunnel number at most two.

\begin{figure}[ht]
\includegraphics*[width=0.7\textwidth]{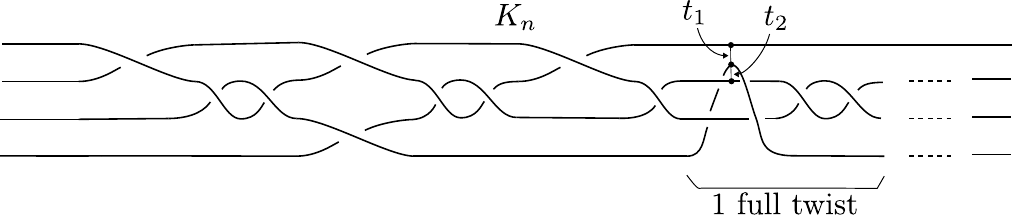}
\caption{An unknotting tunnel system $\{t_1,t_2\}$ for $K_n$.}
\label{fig:tunnel}
\end{figure}

If $K_n$ has tunnel number one, then any resulting manifold by Dehn surgery on $K_n$ has Heegaard genus at most two.
This is impossible by Lemma \ref{lem:graph}.
Hence $K_n$ has tunnel number two.
\end{proof}

\begin{remark}
It is well known that a torus knot has tunnel number one, so is $K_0$. 
For $K_2$, which is \texttt{s682}, it is not hard to see that it admits an unknotting tunnel.
In fact, $t_1$, shown in Figure \ref{fig:tunnel},  gives the unknotting tunnel for $K_2$.
However, $K_1$, the $(2,9)$-cable of $T(2,3)$, has tunnel number two \cite{Bl,MS}.
\end{remark}

\section{Hyperbolicity and braid index}\label{sec:hyp}

In this section, we prove that $K_n$ is hyperbolic when $n\ge 2$.
First, we calculate the Alexander polynomial $\Delta_{K_n}(t)$ of $K_n$.

\begin{lemma}\label{lem:alex}
Let $n\ge 1$.
Then the Alexander polynomial $\Delta_{K_n}(t)$ of $K_n$ is given as
\[
\Delta_{K_n}(t)=(t^{6n+6}-t^{6n+5})+\sum_{i=0}^{n-1}(t^{3n+3i+5}-t^{3n+3i+4}) + t^{3n+3}-\sum_{i=0}^{n-1}(t^{3i+5}-t^{3i+4})-t+1.
\]
\end{lemma}

\begin{proof}
We start from the link $K_0\cup c$ used in Section \ref{sec:Lspace}.
The orientation of each component is shown in Figure \ref{fig:str}.
The multivariable Alexander polynomial $\Delta_{K_0\cup c}(x,y)$
is given as
\[
(x^8-x^5+x^4)y^2+(x^5-x^4+x^3)y+x^4-x^3+1.
\]
(We used Kodama's Knot \cite{K} for the calculation.)

By performing $(-1/n)$-surgery on $c$, the link $K_0\cup c$ is changed into the new link $K_n\cup c_n$,
where $c_n$ is the unknot.  We set $c=c_0$.
Clearly, two links $K_0\cup c$ and $K_n\cup c_n$ have homeomorphic exteriors.
Then the induced isomorphism of the homeomorphism on their homology groups
relates the Alexander polynomials of these two links (\cite{F,Mo2}).

Let $\mu_{K_n}$ and $\lambda_{K_n}$ ($\mu_{c_n}$ and $\lambda_{c_n}$, resp.) be the oriented meridian and longitude of $K_n$ ($c_n$, resp.).
Then $\mu_{K_n}=\mu_{K_0}$, $\mu_{c_n}=\mu_{c}-n\lambda_{c}$ and $\lambda_{c}=3\mu_{K_0}$.
These give $\mu_{c}=\mu_{c_n}+3n\mu_{K_n}$.

Hence we have
\[
\Delta_{K_n\cup c_n}(x,y)=\Delta_{K_0\cup c}(x,yx^{3n}).
\]
Since $\mathrm{lk}(K_n,c_n)=\mathrm{lk}(K_0,c)=3$,
the Torres condition \cite{T} implies
\[
\Delta_{K_n\cup c_n}(x,1)=\frac{x^3-1}{x-1} \Delta_{K_n}(x).
\]
Thus,
\begin{align*}
\Delta_{K_n}(t)&=\frac{t-1}{t^3-1} \Delta_{K_n\cup c_n}(t,1)=\frac{t-1}{t^3-1} \Delta_{K_0\cup c}(t,t^{3n})\\
&=\frac{t-1}{t^3-1}(t^{6n+8}-t^{6n+5}+t^{6n+4}+t^{3n+5}-t^{3n+4}+t^{3n+3}+t^4-t^3+1).
\end{align*}
Here,
\[
t^{6n+8}-t^{6n+5}=t^{6n+5}(t^3-1), \quad t^{6n+4}-t^{3n+4}=t^{3n+4}(t^{3n}-1),
\]
and
\begin{align*}
t^{3n+5}+t^{3n+3}+t^4&=(t^{3n+5}+t^{3n+4}+t^{3n+3})-t^{3n+4}+t^4\\
&=t^{3n+3}(t^2+t+1)-t^4(t^{3n}-1).
\end{align*}
Therefore we have
\begin{align*}
\Delta_{K_n}(t)&=t^{6n+5}(t-1)+t^{3n+4}(t-1)\sum_{i=0}^{n-1}t^{3i}+t^{3n+3}-t^4(t-1)\sum_{i=0}^{n-1}t^{3i}-t+1\\
&=(t^{6n+6}-t^{6n+5})+\sum_{i=0}^{n-1}(t^{3n+3i+5}-t^{3n+3i+4})+t^{3n+3}-\sum_{i=0}^{n-1}(t^{3i+5}-t^{3i+4})\\
& \quad -t+1
\end{align*}
as desired.
\end{proof}

Recall the \textit{formal semigroup\/} $\mathcal{S}$ of an L--space knot $K$ \cite{W}.
It is a set of non-negative integers defined from the formal power series expansion
\[
\frac{\Delta_K(t) }{1-t} = \sum_{s\in \mathcal{S} } t^s \in \mathbb{Z}[[t]].
\]

\begin{lemma}\label{lem:formal}
Let $n\ge 1$.
The formal semigroup of $K_n$ is given as
\begin{align*}
\mathcal{S}&=\{0,4,7,11,\dots, 3n-2,3n+1\}\cup\{ 3n+3\}\cup \{3n+5,3n+6,\dots, 6n-1,6n\}\\
& \quad \cup \{6n+2,6n+3,6n+4\}\cup \{6n+6,6n+7,\dots\}.
\end{align*}
In particular, $\mathcal{S}$ is not a semigroup if $n\ge 2$.
\end{lemma}

\begin{proof}
The first conclusion immediately follows from Lemma \ref{lem:alex}.
If $n\ge 2$, then $4\in \mathcal{S}$ but $8\not\in \mathcal{S}$.
Hence $\mathcal{S}$ is not closed under addition.
\end{proof}

We remark that $\mathcal{S}=\{0,4,6,8,9,10\}\cup \{12,13,\dots\}$ when $n=1$.
This is closed under addition (see \cite[Example 1.10]{BoL} and \cite[Theorem 1.4]{W}).

\begin{lemma}\label{lem:hyp}
If $n\ge 2$, then $K_n$ is hyperbolic.
\end{lemma}

\begin{proof}
By Lemma \ref{lem:formal}, the formal semigroup $\mathcal{S}$ of $K_n$ is not  a semigroup.
Hence $K_n$ is not a torus knot, since the formal semigroup of a torus knot is a genuine semigroup (\cite[Example 1.10]{BoL}).

Suppose that $K_n$ is a satellite knot for a contradiction.
First, $K_n$ is prime, because an L--space knot is prime \cite[Theorem 1.2]{Kr} and Lemma \ref{lem:Lspace}.
Since $K_n$ is the closure of a $4$-braid, it has bridge index at most four.
Hence the companion is a $2$-bridge knot, and the pattern has wrapping number two by \cite{Sc}.

In addition, the companion is an L--space knot, and the pattern is braided \cite[Theorems 1.17 and 7.2]{BM}.
This implies that the companion is a $2$-bridge torus knot  \cite[Theorem 1.5]{OS}, and the pattern is $2$-cabled.
Then $K_n$ is an iterated torus L--space knot, so its formal semigroup is closed under addition \cite[Theorem 1.4]{W}.
This contradicts Lemma \ref{lem:formal}.
Hence $K_n$ is hyperbolic.
\end{proof}


\begin{lemma}\label{lem:braid}
If $n\ge 1$, then 
$K_n$ has braid index four.
\end{lemma}

\begin{proof}
We know that $K_1$ is the $(2,9)$-cable of the trefoil, so it has bridge index four.
Thus $K_1$ has braid index four.

Next, consider $K_2$, which  is \texttt{s682}.
The Morton--Franks--Williams inequality (\cite{FW,Mo}) gives a lower bound on the braid index in terms of its HOMFLY-PT polynomial.
For the $P(v,z)$ version of the HOMFLY-PT polynomial,
let $d_+$ and $d_-$ be the max and min degrees of $v$.
Then $(d_+-d_-)/2+1$ gives the lower bound for the braid index.
For $K_2$, $d_+=24$ and $d_-=18$ by SageMath \cite{Sa}.
This shows that $K_2$ has braid index four.

Assume $n\ge 3$.
If $K_n$ has braid index at most three, then it has tunnel number one as mentioned in Section \ref{sec:intro}.
By Lemma \ref{lem:t2}, $K_n$ has tunnel number two.
Hence it has braid index four.
\end{proof}


\begin{proof}[Proof of Theorem \ref{thm:main2}]
This follows from Lemmas \ref{lem:strong}, \ref{lem:Lspace}, \ref{lem:t2}, \ref{lem:hyp} and \ref{lem:braid}.
\end{proof}



\bibliographystyle{amsplain}

\end{document}